\documentclass[12pt]{amsart}
\usepackage{amssymb,amsfonts,latexsym,amscd}

\usepackage[all,cmtip,2cell]{xy}
\UseAllTwocells
\usepackage{verbatim}

\newtheorem{theorem}{Theorem}[section]
\newtheorem{lemma}[theorem]{Lemma}
\newtheorem{proposition}[theorem]{Proposition}
\newtheorem{corollary}[theorem]{Corollary}
\theoremstyle{definition}

\newtheorem{example}[theorem]{Example}

\newtheorem{remark}[theorem]{Remark}

\newcommand{\id}{\text{id}}

\newcommand{\Hom}{\text{Hom}}

\newcommand{\Aut}{\text{Aut}}

\newcommand{\CoInt}{\text{CoInt}}
\newcommand{\CoInn}{\text{CoInn}}
\newcommand{\Inn}{\text{Inn}}
\newcommand{\Int}{\text{Int}}
\newcommand{\Rep}{\text{Rep}}

\newcommand{\Supp}{\text{Supp}}
\newcommand{\Corep}{\text{Corep}}
\newcommand{\Vect}{\text{Vec}}

\newcommand{\g}{\mathfrak{g}}
\newcommand{\h}{\mathfrak{h}}

\renewcommand{\b}{\mathfrak{b}}

\newcommand{\ot}{\otimes}

\newcommand{\ben}{\begin{enumerate}}
\newcommand{\een}{\end{enumerate}}
\newcommand{\ad}{{\text{Ad}}}

\newcommand{\Lie}{{\text{Lie}}}
\newcommand{\Out}{{\text{Out}}}

\newcommand{\C}{{\mathcal C}}

\begin{document}

\title[Invariant Hopf $2$-cocycles for affine algebraic groups] {Invariant Hopf $2$-cocycles for affine algebraic groups}

\author{Pavel Etingof}
\address{Department of Mathematics, Massachusetts Institute of Technology,
Cambridge, MA 02139, USA} \email{etingof@math.mit.edu}

\author{Shlomo Gelaki}
\address{Department of Mathematics, Technion-Israel Institute of
Technology, Haifa 32000, Israel} \email{gelaki@math.technion.ac.il}

\date{\today}


\begin{abstract}
We generalize the theory of the second invariant cohomology  group $H^2_{\rm inv}(G)$ for finite groups $G$, developed in  \cite{Da2,Da3,GK}, to the case of affine algebraic groups $G$, using the methods of \cite{EG1,EG2,G}. In particular, we show that for connected affine algebraic groups $G$ over an algebraically closed field of characteristic $0$, the map $\Theta$ from \cite{GK} is bijective (unlike for some finite groups, as shown in \cite{GK}). This allows us to compute $H^2_{\rm inv}(G)$ in this case, and in particular show that this group is commutative (while for finite groups it can be noncommutative, as shown in \cite{GK}). 
\end{abstract}

\maketitle

\section{Introduction}

An interesting invariant of a tensor category $\C$ is the group of tensor structures on the identity functor of $\C$ (i.e., the group of isomorphism classes of tensor autoequivalences of $\C$ which act trivially on the underlying abelian category) up to an isomorphism
\cite{Da,Da4,BC,GK,PSV,Sc}. This group is called {\em the second invariant (or lazy) cohomology group} of $\C$ and denoted by 
$H^2_{\rm inv}(\C)$. In particular, if $\C:=\Corep(H)$ is the category of finite dimensional comodules over a Hopf algebra $H$, then $H^2_{\rm inv}(H):=H^2_{\rm inv}(\C)$ is the group of invariant Hopf $2$-cocycles for $H$ modulo the subgroup of coboundaries of invertible central elements of $H^*$. 

In \cite{Da2,GK} the authors study $H^2_{\rm inv}(\C)$ for $\C:=\Corep_k(\mathcal{O}(G))=\Rep_k(G)$, where $G$ is a finite group, and $\mathcal{O}(G)$ is the algebra of functions on $G$ with values in an algebraically closed field $k$. Namely, they define the set $\mathcal{B}(G)$ of pairs $(A,R)$, where $A$ is a normal commutative subgroup in $G$ and $R$ is a non-degenerate $G$-invariant class in $H^2(\widehat{A},k^{\times})$, and a map $\Theta: H^2_{\rm inv}(G):=H^2_{\rm inv}(\C)\to \mathcal{B}(G)$. This allows one to compute $H^2_{\rm inv}(G)$ in many examples, although not always, as in general the map $\Theta$ is neither surjective nor injective. For example, it was proved in \cite{Da2} that if the order of $G$ is coprime to $6$ then  $H^2_{\rm inv}(G)$ is the direct product of $\mathcal{B}(G)$  and the group $\Out_{\rm cl}(G)$ of class-preserving outer automorphisms of $G$.

The goal of this paper is to generalize the theory of \cite{Da2,GK} to the case of affine algebraic groups, using the methods of \cite{EG1,EG2,G}. In particular, we show that for connected affine algebraic groups over a field $k$ of characteristic zero, the map $\Theta$ is, in fact, bijective (so the situation is simpler than for finite groups). This allows us to compute the second invariant cohomology group in this case. 

More specifically, we show in Theorem \ref{propoftheta} that if $k$ has characteristic $0$ and $U$ is a unipotent  algebraic group with $\mathfrak{u}:=\Lie(U)$, then $H^2_{\rm inv}(U)\cong (\wedge^2\mathfrak{u})^{\mathfrak{u}}$. 

Furthermore, let $G$ be any {\em connected} affine algebraic group over $k$ of characteristic $0$. Let $G_{\rm{u}}$ be the unipotent radical of $G$, and let $G_{\rm{r}}:=G/G_{\rm{u}}$. Let $Z$ be the center of $G$, and let $Z_{\rm u},Z_{\rm r}$ be 
the unipotent and reductive parts of $Z$. Let $\g_{\rm u},\mathfrak{z}_{\rm u},\mathfrak{z}_{\rm r}$ be the Lie algebras of $G_{\rm{u}}$, $Z_{\rm u},Z_{\rm r}$,  respectively. Our main result Theorem \ref{mainth} is that one has 
$$H^2_{\rm inv}(G)\cong \Hom(\wedge^2 \widehat Z_{\text{r}},k^{\times})\times (\mathfrak{z}_{\text{r}}\otimes \mathfrak{z}_{\text{u}})\times (\wedge^2\g_{\rm{u}})^G,$$
where $\widehat{Z_{\rm r}}$ is the character group of $Z_{\rm r}$. In particular, this group is commutative (while for finite groups it can be noncommutative, as shown in \cite{GK}). 

The organization of the paper is as follows. Section 2 is devoted to preliminaries. In particular, we recall the definition of the second invariant cohomology group of a Hopf algebra and its basic properties, and study the space $(\wedge^2\g)^\g$ for a Lie algebra $\g$.

In Section 3 we use \cite{G} to describe the structure of $H^2_{\rm inv}(G)$, where $G$ is a {\em commutative} affine algebraic group over $k$ of characteristic $0$ (see Theorem \ref{h2cfinabg}). 

In Section 4 we use \cite{G} to study the {\em support} group $A$ of an invariant Hopf $2$-cocycle $J$ for the function Hopf algebra $\mathcal{O}(G)$ of an arbitrary affine algebraic group over $k$. In particular, we show that $A$ is a closed normal commutative subgroup of $G$ (see Theorem \ref{inv}).

In Section 5 we extend the definitions of the set $\mathcal{B}(G)$ and map $\Theta: H^2_{\rm inv}(G)\to \mathcal{B}(G)$ from \cite{GK} to affine algebraic groups $G$ over $k$, and study the basic properties of $\Theta$ in Theorem \ref{propoftheta}.

In Section 6 we focus on unipotent algebraic groups over $k$ of characteristic $0$, and prove Theorem \ref{propofthetaunip}. 

Finally, Section 7 is devoted to the proof of our main result Theorem \ref{mainth}, which is a generalization of Theorem \ref{propofthetaunip} to arbitrary {\em connected} affine algebraic groups $G$ over $k$ of characteristic $0$.

{\bf Acknowledgements.} We are grateful to Vladimir Popov for communicating to us Lemma \ref{popov}. We thank A. Davydov and M. Yakimov for suggesting the current formulation of Lemma \ref{invcybe}, and A. Davydov for references. The work of P.E.\ was partially supported by the NSF grant DMS-1502244. S.G. thanks the University of Michigan and MIT for their hospitality. 
\section{Preliminaries}

Throughout the paper, unless otherwise specified, we shall work over an algebraically closed field $k$ of arbitrary characteristic.

\subsection{Hopf $2$-cocycles}\label{hopftwococycle} 
Let $H$ be a Hopf algebra over $k$ with multiplication map $m$. An invertible element $J\in (H\ot H)^*$
is called a (right) {\em Hopf $2$-cocycle} for $H$ if it satisfies the two conditions
\begin{align}\label{2coc}
\sum J (a_1b_1,c)J (a_2,b_2)&=\sum J
(a,b_1c_1)J (b_2,c_2),\\ J (a,1) =\varepsilon(a)&=J (1,a)\nonumber
\end{align}
for all $a,b,c\in H$ (see, e.g., \cite{D}).

We have a natural action of the group of invertible elements $(H^*)^{\times}$ on Hopf $2$-cocycles for $H$. Namely, if $J$ is a Hopf $2$-cocycle for $H$ and $x\in (H^*)^{\times}$, then the linear map 
$J^x:H\ot H\to k$ defined by
$$J^x(a,b)=\sum x(a_1b_1)J(a_2,b_2)x^{-1}(a_3)x^{-1}(b_3),\,\,\,a,b\in H,$$
is also a Hopf $2$-cocycle for $H$. We say that two Hopf $2$-cocycles for $H$ are {\em gauge equivalent} if they belong to the same $(H^*)^{\times}$-orbit, and that they are {\em strongly gauge equivalent} if they belong to the same orbit under the action of the group of invertible {\em central} elements of $H^*$.

Given a Hopf $2$-cocycle $J$ for $H$, one can construct a new Hopf algebra $H^{J}$ as follows. As a coalgebra $H^{J}=H$, and the new multiplication $m^{J}$ is given by $m^{J}=J^{-1}*m*J$, i.e., 
\begin{equation}\label{nm}
m^{J}(a\ot b)=\sum J^{-1} (a_1,b_1)a_2b_2J
(a_3,b_3),\,\,\,a,b\in H.
\end{equation}
Equivalently, every Hopf $2$-cocycle $J$ for $H$ defines a tensor structure on the forgetful functor $\Corep_k(H)\to \Vect$, where $\Corep_k(H)$ is the tensor category of finite dimensional comodules over $H$. Recall that if $J$ and $J'$ are gauge equivalent then the Hopf algebras $H^{J}$ and $H^{J'}$ are isomorphic.

Note that if $K$ is a Hopf subalgebra of $H$ and $J$ is a Hopf $2$-cocycle for $H$ then the restriction $\text{res}(J)$ of $J$ to $K$ defines a Hopf $2$-cocycle for $K$.
Also, if $K$ is a Hopf algebra quotient of $H$ and $J$ is a Hopf $2$-cocycle for $K$ then the lifting $\text{lif}(J)$ of $J$ to $H$ defines a Hopf $2$-cocycle for $H$. 

\subsection{Invariant Hopf $2$-cocycles}\label{invhopftwococycle} Let $H$ be a Hopf algebra over $k$ with multiplication map $m$. A Hopf $2$-cocycle $J$ for $H$ is called {\em invariant} (or {\em lazy} in, e.g., \cite[Section 1.4]{GK}) if $m^{J}=m$ (equivalently, if $J*m=m*J$). In particular, when $H$ is cocommutative, every Hopf $2$-cocycle is invariant.

Note that if $K$ is a Hopf subalgebra of $H$ and $J$ is an invariant Hopf $2$-cocycle for $H$ then $\text{res}(J)$ defines an invariant Hopf $2$-cocycle for $K$.
However, if $K$ is a Hopf algebra quotient of $H$ and $J$ is an invariant Hopf $2$-cocycle for $K$ then $\text{lif}(J)$ is not necessarily an invariant Hopf $2$-cocycle for $H$.

It is clear that invariant Hopf $2$-cocycles for $H$ form a group under multiplication, which is denoted by $Z_{\text{inv}}^2(H)$ (= $Z_{\ell}^2(H)$ in \cite{GK}). In fact, we make the following observation.

\begin{lemma}\label{product}
Let $J$ be an invariant Hopf $2$-cocycle for $H$, and let $F$ be a (not necessarily invariant) Hopf $2$-cocycle for $H$. Then $J*F$ is a Hopf $2$-cocycle for $H$. 
\end{lemma}

\begin{proof}
By Eq. (\ref{2coc}) for $J$ and $F$, and the invariance of $J$, we have
\begin{eqnarray*}
\lefteqn{\sum (J*F)(a_1b_1,c)(J*F)(a_2,b_2)}\\
& = & \sum J(a_1b_1,c_1)F(a_2b_2,c_2)J(a_3,b_3)F(a_4,b_4)\\
& = & \sum \left(J(a_1b_1,c_1)J(a_2,b_2)\right ) \left(F(a_3b_3,c_2)F(a_4,b_4)\right )\\
& = & \sum \left(J(a_1,b_1c_1)J(b_2,c_2)\right )\left(F(a_2,b_3c_3)F(b_4,c_4)\right )\\
& = & \sum J(a_1,b_1c_1)F(a_2,b_2c_2)J(b_3,c_3)F(b_4,c_4)\\
& = & \sum (J*F)(a,b_1c_1)(J*F)(b_2,c_2)
\end{eqnarray*}
for all $a,b,c\in H$, as desired. 
\end{proof}

Furthermore, the group $Z_{\text{inv}}^2(H)$ contains a central subgroup $B_{\text{inv}}^2(H)$ (= $B_{\ell}^2(H)$ in \cite{GK}) consisting of all the invariant Hopf $2$-cocycles $\partial(x)$, where $x\in (H^*)^{\times}$ is {\em central} in $H^*$ and
$$\partial(x)(a,b)=\sum x(a_1b_1)x^{-1}(a_2)x^{-1}(b_2),\,\,\,a,b\in H.$$
Following \cite{GK}, we define the quotient group 
\begin{equation}\label{cohom}
H_{\text{inv}}^2(H):=Z_{\text{inv}}^2(H)/B_{\text{inv}}^2(H)
\end{equation}
(= $H_{\ell}^2(H)$ in \cite{GK}). For $J\in Z_{\text{inv}}^2(H)$, let $[J]$ denote its class in $H_{\text{inv}}^2(H)$. By definition, $[J]=[\bar{J}]$ if and only if $J$ and $\bar{J}$ are {\em strongly} gauge equivalent (see Subsection 2.1). 

It is straightforward to verify that a Hopf $2$-cocycle $J$ for $H$ is invariant if and only if it defines a tensor structure on the identity functor on $\C:=\Corep_k(H)$. Hence, $H_{\text{inv}}^2(H)$ is isomorphic to the group $\Aut_0(\C)$ of isomorphism classes of tensor structures on the identity functor of $\C$, i.e., of tensor autoequivalences of $\C$ which act trivially on the underlying abelian category. The second invariant cohomology $\Aut_0(\C)$ for an arbitrary tensor category $\C$ was introduced first by Davydov \cite{Da} and studied also in \cite{PSV}.

\begin{proposition}\cite[Corollary 3.16]{DEN}\label{trivial1} 
$H_{\text{inv}}^2(H)$ is an affine proalgebraic group, and if $H$ is finitely presented, it is an affine algebraic group. \qed
\end{proposition}

\begin{example}\label{sweedler} 
As pointed out in \cite[Section 8]{Da} (see also \cite{Sc}), when $H$ is {\em cocommutative}, the group $H_{\text{inv}}^2(H)$ coincides with Sweedler's second cohomology group $H_{\text{Sw}}^2(H,k)$ of $H$ with coefficients in the algebra $k$ \cite{S}. In particular, 
if $G$ is a group then $H_{\text{inv}}^2(k[G])\cong H^2(G,k^{\times})$ \cite[Theorem 3.1]{S}, and if $\g$ is a Lie algebra then $H_{\text{inv}}^2(U(\g))\cong H^2(\g,k)$ \cite[Theorem 4.3]{S}. 
\end{example}

Recall that a Hopf algebra automorphism of $H$ of the form
$$\ad(x):=x*\id *x^{-1}$$ for some $x\in (H^*)^{\times}$ is called {\em cointernal} (see, e.g., \cite{BC}). For example, if  $x:H\to k$ is an algebra homomorphism (i.e., $x\in H^*_{\rm fin}$ is a grouplike element) then $\ad(x)$ is cointernal and it is called {\em coinner}. The set of all cointernal Hopf algebra automorphisms of $H$ is denoted by $\CoInt(H)$, and the set of all coinner Hopf algebra automorphisms of $H$ is denoted by $\CoInn(H)$. It is easy to see that $\CoInn(H)\subseteq\CoInt(H)$ are normal subgroups of $\Aut_{\text{Hopf}}(H)$ (see, e.g., \cite[Lemma 1.12]{BC}).

The following lemma is dual to \cite[Proposition 1.7(a)]{GK}.

\begin{lemma}\label{trivialgauge} 
Let $J\in Z_{\text{inv}}^2(H)$ and let $x\in (H^*)^{\times}$. Then $J^x\in Z_{\text{inv}}^2(H)$ if and only if $\ad(x)\in \CoInt(H)$.
\end{lemma}

\begin{proof}
By definition, $J^x\in Z_{\text{inv}}^2(H)$ if and only if $J^x*m=m*J^x$, i.e., if and only if $$(x\circ m)*J*(x^{-1}\otimes x^{-1})*m=m*(x\circ m)*J*(x^{-1}\otimes x^{-1}).$$
Hence, since $J$ is invariant, it follows that $J^x\in Z_{\text{inv}}^2(H)$ if and only if
$$(x\circ m)*(x^{-1}\otimes x^{-1})*m=m*(x\circ m)*(x^{-1}\otimes x^{-1}),$$
which is equivalent to
$$
(x^{-1}\otimes x^{-1})*m*(x\otimes x)=(x^{-1}\circ m)*m*(x\circ m).$$
But by definition, the latter is equivalent to $\ad(x)$ being a Hopf automorphism of $H$. We are done. 
\end{proof}

Following \cite[Proposition 1.7]{GK}, we have the following result.

\begin{proposition}\label{trivial}
The quotient group $\CoInt(H)/\CoInn(H)$ acts freely on $H_{\text{inv}}^2(H)$, and the associated map 
\begin{equation}\label{embedding}
\CoInt(H)/\CoInn(H)\to H_{\text{inv}}^2(H),\,\,\,\ad(x)\mapsto (\varepsilon\otimes \varepsilon)^x,
\end{equation} 
defines an embedding of $\CoInt(H)/\CoInn(H)$ as a subgroup of $H_{\text{inv}}^2(H)$.
\end{proposition}

\begin{proof}
By Lemma \ref{trivialgauge}, we have an action of the group $\CoInt(H)$ on $H_{\text{inv}}^2(H)$, given by $\ad(x)\cdot [J]=[J^x]$ for every $\ad(x)\in \CoInt(H)$ and $[J]\in H_{\text{inv}}^2(H)$. This is a well-defined action since if $\ad(x)=\ad(y)$ then $x*y^{-1}$ is central in $H^*$, and hence $[J^x]=[J^y]$.

Now every invariant Hopf $2$-cocycle $J$ for $H$ is fixed by every coinner Hopf automorphism $\ad(x)$ with $x\in G(H^*)$, since $$J^x=(x\circ m)*J*(x^{-1}\otimes x^{-1})=J*(x\circ m)*(x^{-1}\otimes x^{-1})=J.$$ Hence, we get an action of the quotient group $\CoInt(H)/\CoInn(H)$ on $H_{\text{inv}}^2(H)$.

Fix $[J]\in H_{\text{inv}}^2(H)$, and suppose $\ad(x)\in\CoInt(H)$ is such that $\ad(x)\cdot [J]=[J]$. Then $[J]=[J^x]$, hence $J$ and $J^x$ are {\em strongly} gauge equivalent, i.e., $$J=(z\circ m)*(x\circ m)*J*(x^{-1}\otimes x^{-1})*(z^{-1}\otimes z^{-1})$$ for some {\em central} invertible element $z$ in $H^*$. Since $J$ is invariant, $$(z\circ m)*(x\circ m)*J=J*(z\circ m)*(x\circ m),$$ hence we have $\varepsilon\otimes \varepsilon=((z*x)\circ m)*((z*x)^{-1}\otimes (z*x)^{-1})$. In other words, we have $z*x\in G(H^*)$. But $z$ is central, hence $\ad(x)=\ad(z*x)$ is in $\Inn(H)$, which implies the freeness of the action.

Finally, it follows from the above that the map (\ref{embedding}) is injective, and since $(\varepsilon\otimes \varepsilon)^x$ is invariant for every $x\in (H^*)^{\times}$ such that $\ad(x)\in\CoInt(H)$, it is straightforward to verify that $(\varepsilon\otimes \varepsilon)^{xy}=(\varepsilon\otimes \varepsilon)^x*(\varepsilon\otimes \varepsilon)^y$ for every $x,y\in (H^*)^{\times}$ such that $\ad(x),\ad(y)\in\CoInt(H)$, hence (\ref{embedding}) is a group homomorphism.
\end{proof}

\begin{remark}
Let $\C:=\Corep_k(H)$. It is straightforward to verify that $\CoInt(H)/\CoInn(H)$ is the stabilizer of the standard fiber functor of $\C$ in $\Aut_0(\C)$.
\end{remark}

\subsection{Cotriangular Hopf algebras}\label{cothopalg} Recall (see, e.g., \cite[Definition 8.3.19]{EGNO}) that $(H,R)$ is a {\em cotriangular} Hopf algebra if $R\in (H\ot H)^*$ is an invertible element such that 
\begin{itemize}
\item
$\sum R(h_1,g_1)R(g_2,h_2)=\varepsilon(g)\varepsilon(h)$ (i.e., $R^{-1}=R_{21}$),
\item 
$R(h,gl)=\sum R(h_1,g)R(h_2,l)$,
\item
$R(hg,l)=
\sum R(g,l_1)R(h,l_2)$, and
\item
$\sum R(h_1,g_1)g_2h_2=\sum h_1g_1R(h_2,g_2)$
\end{itemize}
for every $h,g,l\in H$. 

Recall that the {\em Drinfeld element} of $(H,R)$ is the grouplike element $u$ in $H^*$, given by $u(h)=\sum R(S(h_2),h_1)$ for every $h\in H$.

Given a Hopf $2$-cocycle $J$ for $H$, $(H^{J},R^J)$ is also cotriangular, where $R^{J}:=J_{21}^{-1}*R*J$. Recall that if $J$ and $\bar{J}:=J^x$ are gauge equivalent then $\ad(x):(H^{J},R^J)\to (H^{\bar{J}},R^{\bar{J}})$ is an isomorphism of cotriangular Hopf algebras.

\begin{proposition}\label{minimal}
Let $(H,R)$ be a cotriangular Hopf algebra, and let $I:=\{a\in H\mid R(b,a)=0,\,b\in H\}$ be the right radical of $R$. Then $I$ coincides with the left radical $\{b\in H\mid R(b,a)=0,\,a\in H\}$ of $R$, and is a Hopf ideal of $H$.
\end{proposition}

\begin{proof}
See \cite[Proposition 2.1]{G}.
\end{proof}

A cotriangular Hopf algebra $(H,R)$ such that $R$ is non-degenerate (i.e., $I=0$) is called {\em minimal}. By Proposition \ref{minimal}, any cotriangular Hopf algebra $(H,R)$ has a unique minimal cotriangular Hopf algebra quotient $H/I$, which we will denote by $(H_{min},R)$. 

Recall that the finite dual Hopf algebra $H^*_{\rm fin}$ of $H$ consists of all the elements in $H^*$ that vanish on some finite codimensional ideal of $H$.  
Note that in the minimal case, the non-degenerate form $R$ defines two injective Hopf algebra maps $R_+,\,R_-:H\hookrightarrow H^*_{\rm fin}$, given by $$R_+(h)(a)=R(h,a)\,\,\,\text{and}\,\,\,R_-(h)(a)=R(S(a),h)$$ for every $a,h\in H$.

\subsection{Invariant solutions to the classical Yang-Baxter equation}\label{drinres}

Recall that a Lie algebra $\h$ over $k$ is called {\em quasi-Frobenius with symplectic form $\omega$} if $\omega\in\wedge^2\h^*$ is a {\em non-degenerate} $2$-cocycle, i.e., $\omega:\h\times \h\to k$ is a non-degenerate skew-symmetric bilinear form satisfying
\begin{equation}\label{2cocy}
\omega([x,y],z)+\omega([z,x],y)+\omega([y,z],x)=0
\end{equation}
for every $x,y,z\in \h$.

Let $\g$ be a Lie algebra over $k$. Recall that an element $r\in \wedge^2 \g$ is a solution to the classical Yang-Baxter equation if
\begin{equation}\label{cybe}
\text{CYB}(r):=[r_{12},r_{13}]+[r_{12},r_{23}]+[r_{13},r_{23}]=0.
\end{equation}
By Drinfeld \cite{Dr}, solutions $r$ to the classical Yang-Baxter equation in $\wedge^2\g$
are classified by pairs $(\h,\omega)$, via $r=\omega^{-1}\in \wedge^2\h$, where $\h\subseteq \g$ is a quasi-Frobenius Lie
subalgebra with symplectic form $\omega$. We will call $\h$ the {\em support} of $r$.

The following result appear in \cite[Lemma 5.5.2]{Da3}.

\begin{lemma}\label{invcybe}
Assume $k$ has characteristic $\ne 2$. Then the components of any element $r\in (\wedge^2\g)^\g$ commute. Hence, any element in $(\wedge^2\g)^\g$ is a solution to the classical Yang-Baxter equation. 
\end{lemma}

\begin{proof}
Let $\h$ be the span of the components of $r$. Since $r$ is $\g$-invariant, $\h$ is a Lie ideal of $\g$. 
Moreover, $\h$ carries a $\g$-invariant symplectic form $\omega=r^{-1}$. Thus, 
$\omega([x,y],z)=-\omega(z,[x,y])=-\omega([z,x],y)$ for every $x,y,z\in \h$. I.e., $\omega([x,y],z)$ is anti-invariant under the cyclic permutation of $x,y,z$. Since this permutation has order $3$, we get $\omega([x,y],z)=0$ for every $x,y,z\in \h$. But since $\omega$ is non-degenerate, $[x,y]=0$ for every $x,y\in \h$, as desired.  
\end{proof}

\begin{proposition}\label{rscentral} 
Assume $k$ has characteristic $\ne 2,3$. If $r,s\in (\wedge^2\g)^\g$ then $$[r,s]=\Delta(z)-z\otimes 1-1\otimes z$$ for some central element $z\in U(\mathfrak{g})$ of degree $3$. Moreover, we have $[r,[r,s]]=[s,[r,s]]=0$.
\end{proposition} 

\begin{proof} 
Let $\mathfrak{a}$ and $\mathfrak{b}$ be the supports of $r$  and $s$, respectively. By Lemma \ref{invcybe}, $\mathfrak{a}$ and $\mathfrak{b}$ are abelian Lie ideals in $\g$. Therefore, $\mathfrak{a}+\mathfrak{b}$ is a Lie subalgebra (in fact, a Lie ideal) of $\g$, which is nilpotent of index $\le 2$. Namely, let 
$\mathfrak{h}:=\mathfrak{a}\cap\mathfrak{b}$, and let $\mathfrak{a}_0$ and $\mathfrak{b}_0$ be some complements of $\mathfrak{h}$ in $\mathfrak{a}$ and $\mathfrak{b}$, respectively (as vector spaces). Then $\mathfrak{a}+ \mathfrak{b}=\mathfrak{a}_0\oplus \mathfrak{b}_0\oplus \mathfrak{h}$. Since $\mathfrak{a},\mathfrak{b}$ are abelian  Lie ideals in 
$\mathfrak{a}+ \mathfrak{b}$, we have that $\mathfrak{h}$ is central in $\mathfrak{a}+ \mathfrak{b}$, and the only nontrivial component of the bracket is 
$[,]:\mathfrak{a}_0\times \mathfrak{b}_0\to \mathfrak{h}$. 

We have $r=r'+r''$, where $r'\in(\mathfrak{a}_0\otimes \mathfrak{h})\oplus 
\wedge^2 \mathfrak{h}$ and $r''\in\wedge^2 \mathfrak{a}_0$. Let $r''=\sum_i x_i\otimes x_i'$ be a shortest presentation of $r''$ in $\wedge^2 \mathfrak{a}_0$. Let $b\in \mathfrak{b}_0$. Since $[\Delta(b),r]=0$, we have $\sum_i ([b,x_i]\otimes x_i'+x_i\otimes [b,x_i'])=0$. Hence, $[b,x_i]=[b,x_i']=0$ for all $i$. Thus, $x_i,x_i'$ are central in $\mathfrak{a}+ \mathfrak{b}$. So 
$[r'',s]=0$, and it suffices to show that $[r',s]=\Delta(z)-z\otimes 1-1\otimes z$ for some central element $z\in U(\mathfrak{g})$ of degree $3$. 

Similarly, we have a decomposition $s=s'+s''$, where the components of $s''$ are central in $\mathfrak{a}+ \mathfrak{b}$. Therefore, $[r',s'']=0$, and it suffices to show that $[r',s']=\Delta(z)-z\otimes 1-1\otimes z$.

Let $\{h_i\}$ be a basis of $\mathfrak{h}$. Then $r'=\sum_i a_i\wedge h_i$ and $s'=\sum_i b_i\wedge h_i$, modulo $\wedge^2 \mathfrak{h}$, where $a_i\in \mathfrak{a}_0$ and  $b_i\in \b_0$. Thus,
$$[r',s']=\sum_{i,j}([a_i,b_j]\otimes h_ih_j+h_ih_j\otimes [a_i,b_j]).$$

Now write $[a_i,b_j]=\sum_k c_{ijk}h_k$. Since $r'$ is $(\mathfrak{a}+ \mathfrak{b})$-invariant, it follows that $\sum_i [a_i,b_j]\wedge h_i=0$ for every $j$, i.e., $\sum_{i,k} c_{ijk}h_k\wedge h_i=0$ for every $j$. Thus, $c_{ijk}$ is symmetric in $i,k$. Similarly, since $s'$ is $(\mathfrak{a}+ \mathfrak{b})$-invariant, $c_{ijk}$ is symmetric in $j,k$. Thus $c_{ijk}$ is symmetric in all three indices. 

Finally, we have 
$$[r,s]=[r',s']=\sum_{i,j,k}c_{ijk}(h_k\otimes h_ih_j+h_ih_j\otimes h_k).$$ Thus we may take $z:=\frac{1}{3}\sum_{i,j,k}c_{ijk}h_ih_jh_k=\frac{1}{6}\text{mult}([r,s])$, where $\text{mult}$ denotes multiplication of components, and obtain that
$$[r,s]=[r',s']=\Delta(z)-z\otimes 1-1\otimes z,$$ as desired. Note that $z$ is central in $U(\g)$, since $r$ and $s$ are $\g$-invariant, and $[r,[r,s]]=[s,[r,s]]=0$ since $[r,h_i]=[s,h_i]=0$ for every $i$.
\end{proof}

\begin{example} Let $\g$ be the Heisenberg Lie algebra with basis $a,b,c$ such that $[a,b]=c$, $[a,c]=[b,c]=0$, $r:=a\wedge c$ and $s:=b\wedge c$. 
Then $[r,s]=c\otimes c^2+c^2\otimes c$, so we may take $z:=c^3/3$. (See \cite[Example 5.5.9]{Da3} for more details.)
\end{example}

Finally, we let $\mathcal{B}(\g)$ denote the set of pairs $(\h,\omega)$ such that $\h\subseteq \g$ is an abelian Lie ideal with a $\g$-invariant symplectic form $\omega$. As mentioned above, we have a bijection $\mathcal{B}(\g)\xrightarrow{\cong} (\wedge^2\g)^\g$, given by $(\h,\omega)\mapsto \omega^{-1}$.

\section{Hopf $2$-cocycles for commutative affine algebraic groups}

In this section we will assume that $k$ has characteristic $0$.

\subsection{Invariant Hopf $2$-cocycles for affine algebraic groups} Let $G$ be a (possibly disconnected) affine algebraic group over $k$, and let $G_{\rm{u}}$ be the unipotent radical of $G$. Recall that we have a split projection $G\to G/G_{\rm{u}}$, i.e., $G=G_{\rm{r}}\ltimes G_{\rm{u}}$ is a semidirect product for a (unique, up to conjugation) closed subgroup $G_{\rm{r}}\cong G/G_{\rm{u}}$ (a ``Levi subfactor") such that $G_{\rm{r}}$ is reductive. 

Equivalently, we have a split exact sequence of Hopf algebras 
$$\mathcal{O}(G_{\rm{r}})\xrightarrow{\iota}\mathcal{O}(G)\xrightarrow{p}\mathcal{O}(G_{\rm{u}})$$
with the Hopf algebra section $\mathcal{O}(G)\xrightarrow{q}\mathcal{O}(G_{\rm{r}})$ corresponding to the inclusion $G_{\rm{r}}\hookrightarrow G$. 

Set $$Z_{\text{inv}}^2(G):=Z_{\text{inv}}^2(\mathcal{O}(G))\,\,\,\text{and}\,\,\,H_{\text{inv}}^2(G):=H_{\text{inv}}^2(\mathcal{O}(G)).$$ Let us say that $J$ is a Hopf $2$-cocycle for $G$ if $J$ is a Hopf $2$-cocycle for $\mathcal{O}(G)$, and that $J$ is {\em minimal} if the cotriangular Hopf algebra $(\mathcal{O}(G)^J,J_{21}^{-1}*J)$ is minimal (see Subsection 2.3).

For $J\in Z_{\text{inv}}^2(G)$, we will call $J_{\rm r}:=\text{lif}(\text{res}(J))$ (with respect to $\iota$ and $q$; see Subsection \ref{hopftwococycle}) the {\em reductive part} of $J$. It is clear that $J_{\rm r}$ is a Hopf $2$-cocycle for $G$. 

\begin{lemma}\label{sspart}
Let $J\in Z_{\text{inv}}^2(G)$ be an invariant Hopf $2$-cocycle for $G$. Then $J_{\rm r}^{-1}$ is a Hopf $2$-cocycle for $G$.
\end{lemma}

\begin{proof}
Since $J$ is invariant, $J^{-1}$ is also a Hopf $2$-cocycle for $G$. Hence, $J_{\rm r}^{-1}=\text{lif}(\text{res}(J^{-1}))$ is a Hopf $2$-cocycle for $G$ too.
\end{proof}

By Lemmas \ref{product} and \ref{sspart}, if $J\in Z_{\text{inv}}^2(G)$ then $J_{\rm u}:=J*J_{\rm r}^{-1}$ is a Hopf $2$-cocycle for $G$. Clearly the restriction of $J_{\rm u}$ to $\mathcal{O}(G_{\rm{r}})$ is trivial (i.e., it defines a fiber functor on $\Rep(G)$, which coincides with the standard one on the semisimple tensor subcategory $\Rep(G/G_{\rm{u}})$). We will call $J_{\rm u}$ the {\em unipotent part} of $J$. We thus obtain

\begin{proposition}\label{uniptw}
For every invariant Hopf $2$-cocycle $J$ for $G$, we have  $J=J_{\rm u}*J_{\rm r}$. \qed
\end{proposition}

\subsection{Hopf $2$-cocycles for commutative affine algebraic groups} Let $A$ be a {\em commutative} affine algebraic group over $k$. Recall that $A\cong A_{\rm{r}}\times A_{\rm{u}}$ is a direct product, and $A_{\rm{r}}\cong T\times A_f$ is a direct product of an algebraic torus $T\cong\mathbb{G}_m^n$ and a finite abelian group $A_f$, and $A_{\rm{u}}\cong\mathbb{G}_a^m$ is an additive group. The diagonalizable closed subgroup $A_{\rm{r}}$ of $A$ is the reductive part of $A$, and the closed subgroup $A_{\rm{u}}$ is the unipotent radical of $A$.

Let $\widehat{A}:=\text{Hom}(A,k^{\times})$ denote the character group of $A$. Recall that $\widehat{T}\cong \mathbb{Z}^n$. We have $\mathcal{O}(A_{\rm{r}})\cong k[\widehat{A_{\rm{r}}}]$, the group algebra of $\widehat{A_{\rm{r}}}\cong\widehat{T}\times\widehat{A_f}$ with its standard Hopf algebra structure, and $\widehat{A}=\widehat{A_{\rm{r}}}$.

Recall also that $\mathcal{O}(A_{\rm{u}})$ is a polynomial algebra on $m$ variables, with its standard Hopf algebra structure, and we have $$\mathcal{O}(A)\cong\mathcal{O}(A_{\rm{r}})\otimes\mathcal{O}(A_{\rm{u}})\cong k[\widehat{A_{\rm{r}}}]\otimes\mathcal{O}(A_{\rm{u}}),$$ as Hopf algebras.

\begin{corollary}\label{decompabelain}
Let $A\cong A_{\rm{r}}\times A_{\rm{u}}$ be a commutative affine algebraic group as above, and let $\mathfrak{a}$, $\mathfrak{a}_{\rm{r}}$ and $\mathfrak{a}_{\rm{u}}$ be the Lie algebras of $A$, $A_{\rm{r}}$ and $A_{\rm{u}}$, respectively. Let 
$J$ be a Hopf $2$-cocycle for $A$. Then $J_{\rm r}$ and $J_{\rm u}$ are (invariant) Hopf $2$-cocycles for $A$, and $J=J_{\rm r}*J_{\rm u}=J_{\rm u}*J_{\rm r}$. Furthermore, we have $[J_{\rm u}]=[\exp(r/2)]$ for a unique element $r$ in \linebreak $\mathfrak{a}\wedge \mathfrak{a}_{\rm{u}}=\mathfrak{a}_{\rm{r}}\otimes \mathfrak{a}_{\rm{u}}\oplus \wedge^2 \mathfrak{a}_{\rm{u}}$. In particular, the Drinfeld element of the cotriangular Hopf algebra $(\mathcal{O}(A),J_{21}^{-1}*J)$ is trivial.
\end{corollary}
 
\begin{proof} 
The first assertion follows from Proposition \ref{uniptw}. The second assertion is a special case of \cite[Theorem 3.2]{EG2}, but here is a direct proof. Consider $\rho:=2\log(J_{\rm u})$ (it is well defined since $J_{\rm u}$ is unipotent). Then $\rho$ is a Hochschild $2$-cocycle of $\mathcal{O}(A)$ with trivial coefficients. Since by the Hochschild-Kostant-Rosenberg theorem, $HH^i(\mathcal{O}(A),k)\cong\wedge^i \mathfrak{a}$,  this isomorphism maps the class of $\rho$ to an element $r\in \wedge^2 \mathfrak{a}$, which is moreover in $\mathfrak{a}\wedge \mathfrak{a}_{\rm{u}}$ since $\rho$ is nilpotent. The claim follows.
\end{proof}

\begin{theorem}\label{h2cfinabg}
Let $A\cong A_{\rm{r}}\times A_{\rm{u}}$ be a commutative affine algebraic group as above, and let $\mathfrak{a}$ and $\mathfrak{a}_{\rm{u}}$ be the Lie algebras of $A$ and $A_{\rm{u}}$, respectively. Then the following hold:

\begin{enumerate}
\item
We have group isomorphisms
$$H_{\text{inv}}^2(A_{\rm{r}})\cong H^2(\widehat{A_{\rm{r}}},k^{\times})\cong\Hom(\wedge^2 \widehat{A_{\rm{r}}},k^{\times}).$$
Furthermore, minimal Hopf $2$-cocycles for $A_{\rm{r}}$ correspond under this isomorphism to non-degenerate alternating bicharacters on $\widehat{A_{\rm{r}}}$.

\item
We have a vector space isomorphism
$$\wedge^2 \mathfrak{a}_{\rm{u}}\xrightarrow{\cong}H_{\text{inv}}^2(A_{\rm{u}}),\,\,\,r\mapsto [\exp(r/2)].$$
Furthermore, minimal Hopf $2$-cocycles for $A_{\rm{u}}$ correspond under this isomorphism to non-degenerate elements  of $\wedge^2 \mathfrak{a}_{\rm{u}}$. 
In particular, $\mathcal{O}(A_{\rm{u}})$ has a minimal Hopf $2$-cocycle if and only if $m$ is even.

\item
We have group isomorphisms
$$H_{\text{inv}}^2(A)\cong H^2(\widehat{A_{\rm{r}}},k^{\times})\times (\mathfrak{a}\wedge \mathfrak{a}_{\rm{u}})\cong\Hom(\wedge^2 \widehat{A_{\rm{r}}},k^{\times})\times (\mathfrak{a}_{\rm{r}}\otimes \mathfrak{a}_{\rm{u}}\oplus \wedge^2 \mathfrak{a}_{\rm{u}}).
$$ 
\end{enumerate}
\end{theorem}

\begin{proof}
(1) and (2) follow from the definitions in a straightforward manner (see \cite[Theorem 5.3 \& Proposition 5.4]{G}). 

(3) follows from (1) and Corollary \ref{decompabelain}.
\end{proof}

\begin{remark}\label{bichar}
The group $H_{\text{inv}}^2(A)$ plays the role of the group of skew-symmetric bicharacters of $A$ when $A_{\rm u}$ is non-trivial.
\end{remark}

\begin{remark}\label{abct}
The group isomorphism $$H^2(\widehat{A_{\rm{r}}},k^{\times})\cong\Hom(\wedge^2 \widehat{A_{\rm{r}}},k^{\times})$$ is given by $J\mapsto R^J:=J_{21}^{-1}J$.
\end{remark}

\begin{remark}
The dimension of an algebraic torus having a minimal Hopf $2$-cocycle need not be even (see \cite[Remark 4.2]{G}). 

Moreover, if $A\cong T\times A_f$ has a minimal Hopf $2$-cocycle then the order of $A_f$ does not have to be a perfect square (contrary to the case of finite groups). For example, let $A=\mathbb{G}_m^2\times \mathbb{Z}/n\mathbb{Z}$, let $q\in k^{\times}$ be a non-root of unity, $\zeta\in k$ a primitive $n$-th root of unity, and let $R$ be the bicharacter of $\widehat{A}$ given by $R((x,y,a),(x',y',a'))=q^{xy'-yx'}\zeta^{xa'-ax'}$. Then $R$ is non-degenerate, hence defines a minimal cotriangular structure on $\mathcal{O}(A)$. It corresponds to a minimal Hopf $2$-cocycle $J$ for $A$ given by $J((x,y,a),(x',y',a'))=q^{xy'}\zeta^{xa'}$. 

This raises a question which finite groups can arise as component groups $H/H^0$ of affine algebraic groups $H$ such that $\mathcal{O}(H)$ has a minimal Hopf $2$-cocycle. Clearly, this can be the product of any group $\Gamma$ of central type with any finite abelian group, by taking $H$ to be the product of $\Gamma$ with a number of copies of the example above for various values of $n$. 
\end{remark}

\begin{example}
Let $A_{\rm{r}}:=\mathbb{G}_m$ and $A_{\rm{u}}:=\mathbb{G}_a$, with Lie algebra bases $x$ and $y$, respectively. Then $r:=x\wedge y$ defines a minimal Hopf $2$-cocycle $\exp(r/2)$ for $A_{\rm{r}}\times A_{\rm{u}}$.
\end{example} 

\begin{example} 
Let $A_{\rm{r}}:=\mathbb{G}_m^2$ and $A_{\rm{u}}:=\mathbb{G}_a$, with Lie algebra bases $\{x_1,x_2\}$ and $y$. Then for every irrational $a$, $r_a:=(x_1+ax_2)\wedge y$ defines a minimal Hopf $2$-cocycle $\exp(r_a/2)$ for $A_{\rm{r}}\times A_{\rm{u}}$, with a $3$-dimensional support. (See \cite[Example 4.13]{G}.)
\end{example}

\begin{remark}\label{rmkbij}
Since a commutative affine algebraic group $A$ over $k$ has only $1$-dimensional irreducible representations, all fiber functors on $\Rep_k(A)=\Corep_k(\mathcal{O}(A))$ preserve dimensions, so all of them arise from Hopf $2$-cocycles for $A$.
\end{remark}

Let us now consider commutative affine algebraic groups over a field $k$ of characteristic $p>0$.

\begin{proposition}\label{h2cfinabgposchar}

\begin{enumerate}
\item
If $A$ is a finite $p$-group then $H_{\text{inv}}^2(A)$
is the trivial group.

\item
If $A$ is a vector group, we have a vector space isomorphism
$H_{\text{inv}}^2(A)\cong\wedge^2 \mathfrak{a}$, where $\mathfrak{a}$ is the Lie algebra of $A$.
\end{enumerate}
\end{proposition}

\begin{proof}
(1) Follows from \cite[Corollary 6.9]{G2}. 

(2) Since in this case $\mathcal{O}(A)\cong U(\mathfrak{a}^*)$, we have $$H_{\text{inv}}^2(A)=H_{\text{inv}}^2(U(\mathfrak{a}^*))=H^2(\mathfrak{a}^*,k)=\wedge^2 \mathfrak{a}$$ (see Example \ref{sweedler}).
\end{proof}

\section{The support of an invariant Hopf $2$-cocycle for an affine algebraic group}\label{suppinvariant}

Let $G$ be an affine algebraic group over $k$ and let $\mathcal{O}(G)$ be its function algebra. Then $(\mathcal{O}(G),\varepsilon\otimes \varepsilon)$ is a cotriangular commutative Hopf algebra.

\begin{theorem}\label{support} 
Let $J$ be a Hopf $2$-cocycle for $G$. 
Then there exist a closed subgroup $H$ of $G$ (called the {\em support} of $J$), determined uniquely up to conjugation, and a minimal Hopf $2$-cocycle $\bar{J}$ for $H$ such that $J$ is gauge equivalent to $\bar{J}$ (viewed as a Hopf $2$-cocycle for $G$).
\end{theorem}

\begin{proof}
For the existence of $H$ and $\bar{J}$, see \cite[Theorem 3.1]{G}.

Now let $H,\bar{H}\subseteq G$ be two closed normal subgroups of $G$, and let $J\in (\mathcal{O}(H)\otimes \mathcal{O}(H))^*$ and $L\in (\mathcal{O}(\bar{H})\otimes \mathcal{O}(\bar{H}))^*$ be two minimal Hopf $2$-cocycles for $H$ and $\bar{H}$, respectively. Suppose $J,L$ are gauge equivalent as Hopf $2$-cocycles for $G$. Then our job is to show that $H,\bar{H}$ are conjugate in $G$. 

By definition, $L=J^x=(x\circ m)*J*(x^{-1}\otimes x^{-1})$ for some element $x\in (\mathcal{O}(G)^*)^{\times}$. Hence, we have
$R^L=(x\otimes x)*R^J*(x^{-1}\otimes x^{-1})$. Then by the minimality of $R^L$ and $R^J$, we have 
\begin{equation}\label{neweq}
(x\otimes x)*(\mathcal{O}(H)\otimes \mathcal{O}(H))^* *(x^{-1}\otimes x^{-1})=(\mathcal{O}(\bar{H})\otimes \mathcal{O}(\bar{H}))^*
\end{equation}
inside the algebra $(\mathcal{O}(G)\otimes \mathcal{O}(G))^*$. It follows that
$$
(x\circ m)*(x^{-1}\otimes x^{-1})=L*(x\otimes x)*J^{-1}*(x^{-1}\otimes x^{-1})\in(\mathcal{O}(\bar{H})\otimes \mathcal{O}(\bar{H}))^*$$
is a symmetric Hopf $2$-cocycle for $\bar{H}$. Therefore, $$(x\circ m)*(x^{-1}\otimes x^{-1})=(y\circ m)*(y^{-1}\otimes y^{-1})$$ for some $y\in\mathcal{O}((\bar{H})^*)^{\times}$. Equivalently, $x*y^{-1}\in\mathcal{O}(G)^*$ is a grouplike element, so $x*y^{-1}=g$ for some $g\in G$. Hence, we have
$$(g\otimes g)*(\mathcal{O}(H)\otimes \mathcal{O}(H))^* *(g^{-1}\otimes g^{-1})=(\mathcal{O}(\bar{H})\otimes \mathcal{O}(\bar{H}))^*$$ inside $(\mathcal{O}(G)\otimes \mathcal{O}(G))^*$, which is equivalent to $gHg^{-1}=\bar{H}$.
\end{proof}

We will denote the conjugacy class of $H$ by $\Supp(J)$, and sometimes write $H=\Supp(J)$ by abuse of notation.

The following result was proved for finite groups in \cite[Theorem 2.4 ]{Da2} (see also \cite[Lemma 4.4(a)]{GK}).

\begin{theorem}\label{inv} 
Let $J\in Z_{\text{inv}}^2(G)$ be an invariant Hopf $2$-cocycle with support $A:=\Supp(J)$. Then $A$ is a closed normal commutative subgroup of $G$.
\end{theorem}
Note that since $A$ is normal, it is well defined as a closed subgroup of $G$, not just up to conjugation.

\begin{proof}
Since $J$ is invariant, it follows that $(\mathcal{O}(G),J_{21}^{-1}*J)$ is a commutative cotriangular Hopf algebra and $(\mathcal{O}(A),J_{21}^{-1}*J)$ is a {\em minimal} commutative cotriangular Hopf algebra. Hence, $\mathcal{O}(A)$ is isomorphic to a Hopf subalgebra of its finite dual Hopf algebra $\mathcal{O}(A)^*_{\rm fin}$ (see Subsection \ref{cothopalg}). Since $\mathcal{O}(A)^*_{\rm fin}$ is {\em cocommutative}, it follows that $\mathcal{O}(A)$ is cocommutative, which is equivalent to $A$ being commutative.

Take arbitrary $g\in G$. By Lemma \ref{trivialgauge}, $J^g$ is also invariant, thus $(\mathcal{O}(G),(J^g)_{21}^{-1}*J^g)$ is a commutative cotriangular Hopf algebra too.  
Now since $J$ is invariant, we have 
\begin{eqnarray*}
\lefteqn{(J^g)_{21}^{-1}*J^g}\\
& = & \left((g\otimes g)*J_{21}^{-1}*(g^{-1}\circ m)\right )*\left((g\circ m)*J*(g^{-1}\otimes g^{-1})\right )\\
& = & (g\otimes g)*J_{21}^{-1}*J*(g^{-1}\otimes g^{-1})\\
& = & J_{21}^{-1}*J,
\end{eqnarray*}
which implies that $gAg^{-1}=A$, as desired. 
\end{proof}

\begin{corollary}\label{idcompreductive}
The connected component of the identity of the support of an invariant Hopf $2$-cocycle $J$ for a reductive algebraic group $G$ over $k$ is a torus. In particular, 
if $G$ is semisimple then the support of $J$ is a finite group. \qed
\end{corollary}

\section{The set $\mathcal{B}(G)$}
Let $G$ be an affine algebraic group over $k$.
Following \cite{Da2} (see also \cite{GK}), we let $\mathcal{B}(G)$ denote the set of pairs $(A,R)$ such that $A$ is a commutative closed normal subgroup of $G$, $(\mathcal{O}(A),R)$ is a minimal cotriangular Hopf algebra with trivial Drinfeld element, and $R$ is $G$-invariant (see Subsection \ref{cothopalg}). We set $e:=(\{1\},\varepsilon\otimes \varepsilon)$.

By Theorem \ref{inv}, every invariant Hopf $2$-cocycle $J$ for $G$ gives rise to an element $(\Supp(J),R^{J})\in\mathcal{B}(G)$, where $R^{J}:=J_{21}^{-1}*J$.

Set $\Int(G):=\CoInt(\mathcal{O}(G))$, and observe that since $G$ is the group of grouplike elements of $\mathcal{O}(G)^*$, we have $\CoInn(\mathcal{O}(G))=\Inn(G)$. Recall that by Proposition \ref{trivial}, $\Int(G)/\Inn(G)$ is a subgroup of $H_{\text{inv}}^2(G)$. 

Following \cite[Section 6]{Da2} (see also \cite[Lemma 4.4(b) \& Theorem 4.5]{GK}), we have the following result.

\begin{theorem}\label{propoftheta}
The following hold:
\begin{enumerate}
\item 
The assignment 
$$
\Theta:H_{\text{inv}}^2(G)\to \mathcal{B}(G),\,\,\,[J]\mapsto (\Supp(J),R^J),
$$
is a well-defined map of sets.

\item
The fibers of $\Theta$ are the left cosets of $\Int(G)/\Inn(G)$ inside $H_{\text{inv}}^2(G)$. In particular, $\Theta$ is injective if and only if $\Int(G)=\Inn(G)$.

\item
$\Theta^{-1}(e)\cong \Int(G)/\Inn(G)$ as groups.
\end{enumerate}
\end{theorem}

\begin{proof}
(1) If $[J]=[\bar{J}]$ then by Theorem \ref{support} and Theorem \ref{inv}, we have $\Supp(J)=\Supp(\bar{J})$ and $R^J=R^{\bar{J}}$. Thus, $\Theta([J])$ is independent of the choice of a representative from $[J]$.

(2) Suppose $\Theta([J])=\Theta([\bar{J}])$. Then $J,\bar{J}$ are gauge equivalent as Hopf $2$-cocycles for $G$, supported on the same commutative closed normal subgroup $A$ of $G$, and $R^{J}=R^{\bar{J}}$. Since $J,\bar{J}$ are invariant, the equality between the $R$-matrices implies that $\bar{J}J^{-1}$ is a symmetric Hopf $2$-cocycle for $G$, hence $\bar{J}=J^x$ for some element $x\in \mathcal{O}(G)^*$ such that $\ad(x)\in \Int(G)$, which defines a unique element in $\Int(G)/\Inn(G)$.
   
(3) Follows from (2). 
\end{proof}

\begin{corollary}\label{commbg} 
If $G$ is a commutative affine algebraic group then $\Theta$ is bijective and therefore induces a group structure on $\mathcal{B}(G)$.
\end{corollary}

\begin{proof}
Follows from Theorems \ref{h2cfinabg}, \ref{propoftheta} and the fact $\Int(G)=\Inn(G)$ is trivial.
\end{proof}

Corollary \ref{commbg} implies that for a fixed closed normal commutative subgroup $A$ of $G$, the set of elements $R$ such that $(A,R)\in \mathcal{B}(G)$ is in bijection with the set of non-degenerate elements in $H_{\text{inv}}^2(A)^G$, where $H_{\text{inv}}^2(A)$ is described in Theorem \ref{h2cfinabg}.

\section{Unipotent algebraic groups}

In this section we will assume that $k$ has characteristic $0$.

\begin{theorem}\label{propofthetaunip}
If $U$ is a unipotent algebraic group over $k$ then $\Theta$ is bijective. Moreover, $H_{\text{inv}}^2(U)\cong (\wedge^2\mathfrak{u})^{\mathfrak{u}}$ as affine algebraic groups, where $\mathfrak{u}:=\Lie(U)$, so in particular $H_{\text{inv}}^2(U)$ is commutative.
\end{theorem}

\begin{proof}
We first show that $\Theta$ is surjective. Let $(A,R)\in\mathcal{B}(U)$. By Theorem \ref{h2cfinabg}, we have $R=\exp(r)$ for some $r\in (\wedge^2 \mathfrak{u})^{\mathfrak{u}}$ (as $A$ is commutative unipotent, i.e., a vector group). Let $J:=\exp(r/2)$. Then $J$ is invariant since $R$ is, and we have $\Theta([J])=(A,R)$.

Next we show that $\Theta$ is injective. By Theorem \ref{propoftheta}, we have to show that $\Int(U)=\Inn(U)$.

Let $\mathfrak{m}$ be the augmentation ideal of $\mathcal{O}(U)^*$. Then $\mathcal{O}(U)^*$ is a complete local ring with unique maximal ideal $\mathfrak{m}$. Thus, if $\alpha$ is a cointernal automorphism of $U$ then $\alpha$ is given by conjugation by an element $a$ in $1+\mathfrak{m}$. Hence the element $\log(a)$ defines a derivation of $\mathfrak{u}$, given by $d(x)=[\log(a),x]$, and it suffices to show that this derivation is inner: $d(x)=[u,x]$ for some $u\in \mathfrak{u}$. Then we may take $a=\exp(u)\in U$, and get that $\alpha$ is inner.
 
The derivation $d$ defines a class $[d]\in H^1(\mathfrak{u},\mathfrak{u})$, and its image in $H^1(\mathfrak{u},\mathcal{O}(U)^*)$ is zero (as $d(x)=[\log(a),x]$, where $\log(a)\in \mathcal{O}(U)^*$). Thus, it suffices to show that $\mathfrak{u}$ is a direct summand in $\mathcal{O}(U)^*$ as a $\mathfrak{u}$-module under the adjoint action.  
In other words, we need to show that the surjective morphism $b:\mathcal{O}(U)\to\mathfrak{u}^*$, given by $\beta\mapsto d\beta(1)$, splits as a morphism of $\mathfrak{u}$-modules, i.e., there exists a $\mathfrak{u}$-morphism $\phi:\mathfrak{u}^*\to \mathcal{O}(U)$ such that $b\circ \phi=\id$. The morphism $\phi$ can be viewed as an element of $\mathcal{O}(U)\otimes \mathfrak{u}$, i.e., as a regular function $\phi:U\to\mathfrak{u}$. Thus it suffices to show that there exists a regular $\mathfrak{u}$-invariant function $\phi:U\to\mathfrak{u}$ such that $d\phi(1)=\id$. But in the unipotent case we can simply take $\phi(X)=\log(X)$. The proof that $\Theta$ is injective is complete.

Finally, let us show that the bijection $$\Theta^{-1}\circ \exp:(\wedge^2\mathfrak{u})^{\mathfrak{u}}\to H_{\text{inv}}^2(U),\,\,\,r\mapsto [J_r],\text{ where } J_r:=\exp(r/2),$$ is a group homomorphism. By Proposition \ref{rscentral}, $[r,[r,s]]=[s,[r,s]]=0$ for any $r,s\in (\wedge^2\mathfrak{u})^{\mathfrak{u}}$. So we have $J_r*J_s=J_{r+s}\exp([r,s]/8)$, i.e., $J_r*J_s$ is gauge equivalent to $J_{r+s}$ by the gauge transformation given by the element $\exp(z/8)$, where $z$ is as in Proposition \ref{rscentral}.

The proof of the theorem is complete. 
\end{proof}

\begin{remark}
Theorem \ref{propofthetaunip} is a special case of Theorem \ref{mainth}, but its proof is simpler and we decided to give it separately.
\end{remark}

\begin{example}\label{heisen} 
Let $U$ be the $3-$dimensional Heisenberg group. Then $\mathfrak{u}$ has a basis $a,b,c$ such that $c$ is central and $[a,b]=c$. It is straightforward to verify that $(\wedge^2\mathfrak{u})^{\mathfrak{u}}=\{(\alpha a+\beta b)\wedge c\mid \alpha,\beta\in k\}$, and by Corollary \ref{propofthetaunip}, $$H_{\text{inv}}^2(U)\cong\{\exp((\alpha a+\beta b)\wedge c) \mid \alpha,\beta\in k\}\cong \mathbb{G}_a^2.$$
\end{example}

\section{Connected affine algebraic groups}

In this section we will assume that $k$ has characteristic $0$.

\subsection{The bijectivity of $\Theta$}

\begin{proposition}\label{directsummand}
Let $G$ be a connected affine algebraic group over $k$ with Lie algebra $\g$. Then $\g$ is a direct summand of $\mathcal{O}(G)^*$ as a $\g$-module (under the adjoint action).
\end{proposition}

\begin{proof}
We will show that the surjective morphism $b:\mathcal{O}(G)\to\g^*$, given by $\beta\mapsto d\beta (1)$, splits as a morphism of $\g$-modules, i.e., there exists a $\g$-morphism $\phi:\g^*\to \mathcal{O}(G)$ such that $b\circ\phi=\id$. The morphism $\phi$ can be viewed as an element of $\mathcal{O}(G)\otimes \g$, i.e., as a regular function $\phi: G\to \g$. Thus it suffices to prove the following lemma, communicated to us by Vladimir Popov. 

\begin{lemma}\label{popov}
There exists a regular $G$-invariant function $\phi: G\to \g$ such that $d\phi(1)=\id$.
\end{lemma}

\begin{proof}
This almost follows from \cite[Theorem 10.2 and Proposition 10.5]{LPR}. More precisely, the map $\phi:=\gamma_{\mathfrak{c}}\circ \theta\circ\gamma_{C}^{-1}$ in \cite[2nd line of p. 961]{LPR} actually is defined at $1$ and moreover is \'etale at $1$ (i.e., its differential is an isomorphism). Indeed, $\theta$ is \'etale at $[1,1]$ (because $\epsilon$ is \'etale at $1$ by \cite[Lemma 10.3]{LPR}, and $\tau$ is an isomorphism \cite[4.3]{LPR}), and by \cite[(3.1)]{LPR}, $\gamma_{\mathfrak{c}}$ and $\gamma_{C}$ are \'etale at $[1,0]$ and $[1,1]$, respectively. 

So by multiplying $\phi$ by the inverse of $d\phi(1)$, we can make $d\phi(1)=\id$. 
The remaining problem is that $\phi$ is only a rational map. To make it regular, we use \cite[Proposition 10.6 and Lemma 10.7]{LPR} to construct a regular function $f$ on $G$ such that $f(1)=1$, and $f\circ\phi$ is regular. Then $d(f\circ\phi)(1)=d\phi(1)=\id$, so we may replace $\phi$ with $f\circ\phi$. We are done.
\end{proof}

This completes the proof of the proposition.
\end{proof}

\begin{proposition}\label{reducinj}
Let $G$ be a connected affine algebraic group over $k$. Then $\Int(G)=\Inn(G)$.
\end{proposition}

\begin{proof}
Assume first that $G$ is reductive. Then it follows from the well known description of the automorphism group of $G$ that an automorphism of $G$ is inner if and only if it does not permute the irreducible representations of $G$. Hence, the claim follows in this case.

Consider now the general case. Suppose $x\in(\mathcal{O}(G)^*)^{\times}$ is such that $\alpha:=\ad(x)$ is an automorphism of $G$. Let $G_{\rm u}$ be the unipotent radical of $G$, let $G_{\rm r}:=G/G_{\rm u}$, and let $\psi:\mathcal{O}(G)^*\to\mathcal{O}(G_{\rm r})^*$ be the corresponding algebra surjection. Then $\psi(x)$ implements an automorphism of $G_{\rm r}$, which is inner by the above argument. Hence, multiplying $x$ by an element of $G$, we are reduced to the situation when $\psi(x)=1$. Let $I$ be the kernel of $\psi$, then $x\in 1+I$. 

Thus we can define the element 
$\log(x):=\sum (-1)^{n-1}(x-1)^n/n\in I$ (as any series $\sum_n a_n$, where $a_n\in I^n$, converges in the topology of $\mathcal{O}(G)^*$). We have $[\log(x),y]=d(y)$, where $d:=\log(\alpha)$ is a derivation of $\g:=\Lie(G)$. Thus $\log(x)$ defines a class in $H^1(\g,\g)$ which becomes trivial in $H^1(\g,\mathcal{O}(G)^*)$ (where $\g$ acts on $\mathcal{O}(G)^*$ by the adjoint action). Therefore, since by Proposition \ref{directsummand}, $\g$ is a direct summand of $\mathcal{O}(G)^*$ as a $\g$-module, the class of $\log(x)$ in $H^1(\g,\g)$ is zero, i.e., $[\log(x),y]=[b,y]$ for all $y\in \g$ and some element $b\in \g$. Then it is easy to see that $\ad(x)y=\ad(\exp(b))y$ for all $y\in \g$, hence $xgx^{-1}=\exp(b)g\exp(b)^{-1}$ for all $g\in G$. Thus, $\ad(x)$ is an inner automorphism, as desired.
\end{proof}

\begin{theorem}\label{conntheta}
Let $G$ be a connected affine algebraic group over $k$. Then $\Theta$ is bijective.
\end{theorem}

\begin{proof}
By Theorem \ref{propoftheta} and Proposition \ref{reducinj}, $\Theta$ is injective.

We now show that $\Theta$ is surjective. Let $(A,R)\in\mathcal{B}(G)$. We have $A=A_{\rm{r}}\times A_{\rm{u}}$, the product of the reductive and unipotent parts. Since $\Aut(A_{\rm{r}})$ is discrete, $G$ acts on $A_{\rm {r}}$ trivially, i.e., $A_{\rm {r}}$ is central in $G$. Let $\mathfrak{a}:=\Lie(A)$, $\mathfrak{a}_{\rm{u}}:=\Lie(A_{\rm{u}})$. By Theorem \ref{h2cfinabg}, we have $R=R_{\rm{r}}*R_{\rm{u}}$, where $R_{\rm{u}}=\exp(s)$ for some $s\in \mathfrak{a}\wedge \mathfrak{a}_{\rm{u}}$. Let $J_{\rm{u}}:=\exp(s/2)$, and let $J_{\rm{r}}$ be any Hopf $2$-cocycle for $A_{\rm{r}}$ such that $(J_{\rm{r}})_{21}^{-1}*J_{\rm{r}}=R_{\rm{r}}$. Then $J_{\rm{r}}$ is invariant since $A_{\rm {r}}$ is central in $G$. So setting $J:=J_{\rm{r}}*J_{\rm{u}}$, we see that $J$ is $G$-invariant and $\Theta([J])=(A,R)$. 
\end{proof}

\begin{remark}
If $G$ is a finite group of odd order, it follows from \cite[Lemma 6.1]{Da2} that $\Theta$ is surjective. (See also \cite[Corollary 4.6]{GK}.) However, there exist finite groups of even order for which $\Theta$ is not surjective (see \cite[7.3]{GK})
\end{remark}

%


\subsection{The structure of $H_{\text{inv}}^2(G)$ for nilpotent $G$ with a commutative reductive part}
Let $G:=G_{\rm{r}}\times G_{\rm{u}}$ be a (possibly disconnected) {\em nilpotent} affine algebraic group over $k$, with a {\em commutative reductive part}. Let $\g_{\rm{u}}$ be the Lie algebra of $G_{\rm{u}}$, and let $\mathfrak{z}_{\rm{u}}$ be the center of $\g_{\rm{u}}$. Let $\g_{\rm{r}}$ be the Lie algebra of $G_{\rm{r}}$. Then $\g:=\g_{\rm{r}}\oplus \g_{\rm{u}}$ is the Lie algebra of $G$. Note that $\g_{\rm{r}}$ is central in $\g$.

\begin{lemma}\label{rnilpotent}
Let $r\in \g\wedge \g_{\rm{u}}$ be a $\g$-invariant element. Write $r=r'+r''$, where $r'\in \g_{\rm{r}}\otimes \g_{\rm{u}}$ and $r''\in \wedge^2 \g_{\rm{u}}$. Then $r'\in \g_{\rm{r}}\otimes \mathfrak{z}_{\rm{u}}$ and $r''\in (\wedge^2\g_{\rm{u}})^{\g_{\rm{u}}}$.
\end{lemma}

\begin{proof}
Let $b\in \g_{\rm{u}}$. We have $[\Delta(b), r'+r'']=0$, $[\Delta(b), r'']\in \wedge^2 \g_{\rm{u}}$, and $[\Delta(b), r']=[1\otimes b,r']\in \g_{\rm{r}}\otimes \g_{\rm{u}}$. Thus $[\Delta(b), r'']=[\Delta(b), r']=0$, and it follows that $r'\in \g_{\rm{r}}\otimes \mathfrak{z}_{\rm{u}}$. 
\end{proof} 

\begin{lemma}\label{decompnilpotent}
Let $G=G_{\rm{r}}\times G_{\rm{u}}$ be a nilpotent affine algebraic group with a commutative reductive part as above. The following hold:

\begin{enumerate}
\item
$\Int(G)=\Inn(G)$, so $\Theta:H_{\text{inv}}^2(G) \to \mathcal{B}(G)$ is injective.

\item
The group $H_{\text{inv}}^2(G)$ is commutative and naturally isomorphic to 
$\Hom(\wedge^2 \widehat{G_{\rm{r}}},k^{\times})\times (\mathfrak{g}_{\text{r}}\otimes \mathfrak{z}_{\text{u}})\times (\wedge^2\g_{\rm{u}})^{\g_{\rm u}}$.
\end{enumerate}
\end{lemma}
 
\begin{proof}
(1) Every $\gamma\in \Int(G)$ acts trivially on $G_{\rm{r}}$ (since $G_{\rm{r}}$ is central in $G$), and defines an internal automorphism of $G_{\rm{u}}$ by restriction. Hence,  by Theorem \ref{propofthetaunip}, $\gamma$ is inner.

(2) We have a natural group homomorphism $$\psi:\Hom(\wedge^2 \widehat{G_{\rm{r}}},k^{\times})\times (\mathfrak{g}_{\text{r}}\otimes \mathfrak{z}_{\text{u}})\times (\wedge^2\g_{\rm{u}})^{\g_{\rm u}}\to H_{\text{inv}}^2(G),$$
given by $\psi(K,r',r'')=[K*\exp(r/2)]$, where $r:=r'+r''$ 
and $K$ is viewed as an element in $H_{\text{inv}}^2(G_{\rm{r}})=\Hom(\wedge^2 \widehat{G_{\rm{r}}},k^{\times})$. Thus, we have $(\Theta\circ \psi)(K,r',r'')=(A,R)$, where $A$ is the corresponding commutative normal subgroup to $R:=K_{21}^{-1}*K*\exp(r)$. It is clear that $\Theta\circ \psi$ is injective, hence $\psi$ is injective as well.

Now let us show that $\psi$ is surjective. By Part (1), for this it suffices to prove that $\Theta\circ \psi$ is surjective. Let $(A,R)\in\mathcal{B}(G)$. Write $A=A_{\rm r}\times A_{\rm u}$, and let $\mathfrak{a}_{\rm r}$, $\mathfrak{a}_{\rm u}$ be the Lie algebras of $A_{\rm r}$, $A_{\rm u}$. Then Theorem \ref{h2cfinabg} implies that $R=R_{\rm r}*R_{\rm u}$, where $R_{\rm r}=K_{21}^{-1}*K$ and $R_{\rm u}=\exp(r)$ for some $r\in \mathfrak{a}_{\rm r}\otimes \mathfrak{a}_{\rm u}\oplus \wedge^2 \mathfrak{a}_{\rm u}$. Since $A_{\rm r}\subset G_{\rm r}$ is central in $G$, we have that $K$ and $r$ are $G$-invariant. Hence, by Lemma \ref{rnilpotent}, $r=r'+r''$, where $r'\in \g_{\rm{r}}\otimes \mathfrak{z}_{\rm{u}}$ and $r''\in (\wedge^2\g_{\rm{u}})^{\g_{\rm u}}$.
Thus, $(A,R)=(\Theta\circ \psi)(K,r',r'')$, i.e., $\psi$ is an isomorphism, as desired.
\end{proof}

\subsection{The structure of $H_{\text{inv}}^2(G)$ for connected $G$} 
Let $G$ be a connected affine algebraic group over $k$. Let $G_{\rm{u}}$ be the unipotent radical of $G$, $\g_{\rm{u}}$ its Lie algebra, and $G_{\rm{r}}:=G/G_{\rm{u}}$. Let $Z$ be the center of $G$, $Z_{\text{r}}$  its reductive part and $Z_{\text{u}}$ its unipotent part (so, $Z=Z_{\text{r}}Z_{\text{u}}$). Let $\mathfrak{z}$, $\mathfrak{z}_{\text{r}}$ and $\mathfrak{z}_{\text{u}}$ be the Lie algebras of $Z$, $Z_{\text{r}}$ and $Z_{\text{u}}$, respectively. 

\begin{theorem}\label{mainth} 
Let $G$ be a connected affine algebraic group over $k$. Then the group $H_{\text{inv}}^2(G)$ is commutative and is naturally isomorphic to 
$$\Hom(\wedge^2 \widehat Z_{\text{r}},k^{\times})\times (\mathfrak{z}_{\text{r}}\otimes \mathfrak{z}_{\text{u}})\times (\wedge^2\g_{\rm{u}})^G.$$
\end{theorem} 

\begin{proof} 
We have a natural group homomorphism $$\psi:\Hom(\wedge^2 \widehat Z_{\text{r}},k^{\times})\times (\mathfrak{z}_{\text{r}}\otimes \mathfrak{z}_{\text{u}})\times (\wedge^2\g_{\rm{u}})^G\to H_{\text{inv}}^2(G),$$
given by $\psi(K,r',r'')=[K*\exp(r/2)]$, where $r:=r'+r''$ 
and $K$ is viewed as an element in $H_{\text{inv}}^2(Z_{\rm{r}})=\Hom(\wedge^2 \widehat{Z_{\rm{r}}},k^{\times})$.
Thus, we have $(\Theta\circ \psi)(K,r',r'')=(A,R)$, where $A$ is the corresponding commutative normal subgroup to $R:=K_{21}^{-1}*K*\exp(r)$. It is clear that $\Theta\circ \psi$ is injective, hence $\psi$ is injective as well.
  
Let us now prove that $\psi$ is surjective. By Theorem \ref{conntheta}, $\Theta$ is bijective, so it suffices to show that $\Theta\circ \psi$ is surjective. Let $(A,R)\in\mathcal{B}(G)$. Write $A=A_{\rm r}\times A_{\rm u}$.
Then $A$, $A_{\text{r}}$ and $A_{\text{u}}$ are normal in $G$. Hence, $A_{\text{r}}$ is central in $G$ (since the group of automorphisms of $A_{\text{r}}$ is discrete, while $G$ is connected). Hence, $A\subset G':=Z_{\text{r}}G_{\rm{u}}$, and $(A,R)$ is an element of $\mathcal{B}(G')$. 
Since $G'$ is nilpotent with a {\em commutative} reductive part, it follows from Lemma \ref{decompnilpotent} that $(A,R)=(\Theta\circ \psi)(K,r',r'')$, where $K$ is a Hopf $2$-cocycle for $Z_{\rm r}$, $r'\in \mathfrak{z}_{\rm{r}}\otimes \g_{\rm{u}}$, and $r''\in (\wedge^2\g_{\rm{u}})^{\g_{\rm u}}$. Moreover, $R$ is $G$-invariant, hence so are $r'$, $r''$. Therefore, $r'\in \mathfrak{z}_{\rm{r}}\otimes \mathfrak{z}_{\rm{u}}$ and $r''\in (\wedge^2\g_{\rm{u}})^G$, i.e., $\psi$ is an isomorphism, as desired.
\end{proof}

\begin{remark}
If $G$ is a reductive algebraic group over $k$, Theorem \ref{mainth} says that $H^2_{inv}(G)=H^2(\widehat{Z},k^{\times})=\Hom(\wedge^2\widehat {Z},k^{\times})$, which also follows for $k=\mathbb{C}$ from  \cite[Theorem 7]{NT}. In particular, if $G$ is semisimple then $H_{\text{inv}}^2(G)$ is a finite group, and if the center of $G$ is cyclic then $H_{\text{inv}}^2(G)$ is the trivial group.
\end{remark}

\begin{remark}
The construction of Hopf $2$-cocycles for affine algebraic groups over $k$ of characteristic $0$ from solutions to the classical Yang-Baxter equation was introduced in \cite{EG1} (see, e.g., \cite[Theorem 5.5]{EG1}), following a method of Drinfeld. The map $r\mapsto J_r:=\exp(r/2)$ used throughout this paper is a special case of this construction in the case when $r$ spans an abelian Lie algebra.
\end{remark}

\end{document}